\documentclass[a4paper,10pt]{article}
\usepackage{mathtext}
\usepackage[T1,T2A]{fontenc}
\usepackage[utf8]{inputenc}
\usepackage[english]{babel}
\usepackage{amsmath}
\usepackage{amsfonts}
\usepackage{amssymb}
\usepackage{mathrsfs}
\usepackage{amsthm}

\newtheorem*{Def}{Definition}
\newtheorem{Th}{Theorem}
\newtheorem{Le}[Th]{Lemma}

\newtheorem{Rem}[Th]{Remark}

\numberwithin{equation}{section}

\newcommand{\R}{\mathbb{R}}

\newcommand{\fdot}{\,\cdot\,}
\newcommand{\tb}{\tilde{b}}
\newcommand{\tI}{\tilde{\I}}
\newcommand{\IA}{\I^{\mathcal{A}}}

\DeclareMathOperator{\I}{I}
\DeclareMathOperator{\J}{J}
\newcommand{\A}{\mathcal{A}}
\newcommand{\F}{\mathcal{F}}

\newcommand{\E}{\mathbb{E}}

\renewcommand{\leq}{\leqslant}
\renewcommand{\geq}{\geqslant}

\newcommand{\eq}[1]{\begin{equation}{#1}\end{equation}}

\newcommand{\alg}[1]{\begin{align}{#1}\end{align}}

\newcommand{\set}[2]{\{{#1}\mid{#2}\}}

\def\kratno{\mathrel{\smash{\lower.5ex\hbox{$\vdots$}}}}

\title{Fractional integration for irregular martingales\thanks{Supported by the Russian Science Foundation grant N. 19-71-10023.}}
\author{Dmitriy Stolyarov \and Dmitry Yarcev}
\begin{document}
\maketitle
\begin{abstract}
We suggest two versions of the Hardy--Littlewood--Sobolev inequality for discrete time martingales. In one version, the fractional integration operator is a martingale transform, however, it may vanish if the filtration is excessively irregular; the second version lacks the martingale property while being analytically meaningful for an arbitrary filtration.
\end{abstract}

\section{Martingale fractional integration}
The classical Hardy--Littlewood--Sobolev inequality plays an important role in analysis, see e.g. Chapter~$5$ in~\cite{Stein1970}. It says that the Riesz potential of order~$\alpha$, i.e. the Fourier multiplier with the symbol~$|\fdot|^{-\alpha}$, maps~$L_p(\R^d)$ to~$L_q(\R^d)$ provided~$\frac{1}{p} - \frac{1}{q} = \frac{\alpha}{d}$ and~$1 < p < q < \infty$. The Riesz potential is often referred to as the fractional integration operator. As many other objects in Harmonic Analysis, the fractional integration has probabilistic interpretation. In~\cite{Watari1964}, Watari transferred this notion to dyadic martingales and proved the corresponding version of the Hardy--Littlewood--Sobolev inequality; see also~\cite{ChaoOmbe1985} and~\cite{Cruz-UribeMoen2013}  for related results. In~\cite{NakaiSadasue2012}, these ideas were generalized to the setting of regular filtrations and martingales. 

Consider the atomless probability space~$(\Omega,\Sigma,P)$ equipped with the filtration~$\F = \{\F_n\}_{n=0}^{\infty}$. Assume that each~$\sigma$-algebra~$\F_n$ is generated by at most countable number of atoms and denote by~$\A(\F_n)$ the set of atoms of~$\F_n$. We further assume that the filtration~$\F$ separates the points of~$\Omega$ in the sense that~$\cup_n \F_n$ generates~$\Sigma$. The filtration~$\F$ is called regular provided there exists~$\rho > 0$ such that any atoms~$w\in \A(\F_n)$ and~$v \in \A(\F_{n+1})$ such that~$v \subset w$ satisfy~$P(v) \geq \rho P(w)$. This is the same as to say that the inequality~$\rho F_{n+1}\leq F_n$ is true for any non-negative martingale~$F$ adapted to~$\F$. We denote the martingale differences of~$F$ by~$\Delta F_n$, i.e.~$\Delta F_n = F_n - F_{n-1}$ whenever~$n \geq 1$. It is also convenient to introduce auxiliary functions~$b_n$:
\eq{\label{bn}
b_n =  \sum\limits_{w\in \A(\F_n)} P(w)\chi_{w}.
}
Consider the operator~$\I_{\alpha}$,~$\alpha \in (0,1)$, acting on martingales by the rule
\eq{\label{NakaiSadasueOperator}
\I_{\alpha}[F] = \sum\limits_{n=1}^{\infty} b_{n-1}^\alpha \Delta F_n.
}
Note that this operator is a martingale transform in the sense that if we truncate the summation and consider the sequence
\eq{
\I_{\alpha}[F]_N =  \sum\limits_{n=1}^{N} b_{n-1}^\alpha \Delta F_n,
}
then this sequence forms a martingale adapted to~$\F$. Nakai and Sadasue in~\cite{NakaiSadasue2012} proved that~$\I_{\alpha}$ is~$L_p\to L_q$ continuous provided~$\frac{1}{p} - \frac{1}{q} = \alpha$,~$1 < p < q<\infty$, and~$\F$ is a regular filtration (see Theorem~$5.1$ in~\cite{NakaiSadasue2012}). As we have already said, the dyadic case (i.e. when each atom in~$\A(\F_n)$ is split into two atoms of equal probability in~$\A(\F_{n+1})$) had been already considered by Watari in~\cite{Watari1964}; see~\cite{ApplebaumBanuelos2014} how to link rigorously the dyadic martingale Hardy--Littlewood--Sobolev inequality to the classical Euclidean version, that paper also contains information on generalization to the setting of semigroups of operators.  

It is easy to see by applying~$\I_{\alpha}$ to single step martingales (i.e. martingales for which only one martingale difference~$\Delta F_n$ is non-zero) that without the regularity assumption the operator~$\I_{\alpha}$ may not be continuous as an~$L_p \to L_q$ operator (see Remark~\ref{Sharpness} below). On the other hand, the papers~\cite{BanuelosOsekowski2017} and~\cite{Osekowski2014} provide results about the sharp constants in weak-type inequalities for operators similar to~$\I_{\alpha}$ on uniform filtrations (i.e. each atom in~$\A(\F_n)$ is split into~$m$ atoms of equal probability in~$\A(\F_{n+1})$; here~$m$ is independent of~$w$ and~$n$); the corresponding constants appear to be uniformly bounded with respect to~$m$. This hints there must be a reasonable (i.e.~$L_p \to L_q$ continuous) generalization of the operator~\eqref{NakaiSadasueOperator} to the setting of irregular filtrations. 

Consider the modified functions~\eqref{bn}:
\eq{
\tb_n(x) = \inf \set{b_n(y)}{y \in \omega},\quad x\in \omega \in \A(\F_{n-1}).
}
Note that~$\tb_n$ is~$\F_{n-1}$-measurable and for regular filtrations~$\tb_n$ is comparable to~$b_{n-1}$. Thus, the operator
\eq{\label{tI}
\tI_\alpha[F] = \sum\limits_{n=1}^{\infty} \tb_{n}^\alpha \Delta F_n
}
generalizes~\eqref{NakaiSadasueOperator}. Note that this operator is also a martingale transform. Unfortunately~$\tI_{\alpha}$ vanishes if~$w\in \A(\F_n)$ contains infinitely many atoms of~$\F_{n+1}$. The operator
\eq{\label{IA}
\IA_\alpha[F] = \sum\limits_{n=1}^{\infty} b_{n}^\alpha \Delta F_n
}
is no longer a martingale transform; this operator is more interesting from the analytic point of view.  In the case where~$\F$ is uniform, the three operators~$\I_{\alpha}$,~$\tI_{\alpha}$, and~$\IA_{\alpha}$ are multiples of each other.
\begin{Th}\label{IrregularHLS}
Let~$\alpha \in (0,1)$\textup,~$1< p < q < \infty$\textup, and~$\frac{1}{p} - \frac{1}{q} = \alpha$. The operators~$\tI_{\alpha}$ and~$\IA_{\alpha}$ map~$L_p$ martingales to~$L_q$ continuously.
\end{Th}
\begin{Rem}\label{Sharpness}
The theorem above is sharp in the following sense. Let~$w\in \A(\F_n)$ and~$v\in\A(\F_{n+1})$ be such that~$v\subset w$. Consider the martingale~$F$ given by the rule
\eq{
F_{n+1}(x) = \begin{cases} \frac{1}{P(v)},\quad &x\in v;\\
-\frac{1}{P(w) - P(v)},\quad &x\in w\setminus v;\\
0,\quad &x\notin w,
\end{cases}
}
with~$F_m = 0$ when~$m < n+1$ and~$F_m = F_{n+1}$ when~$m \geq n+1$. The quantities~$\|\IA_\alpha[F]\|_{L_q}$ and~$\|F\|_{L_p}$ are comparable provided~$P(v) \leq \frac12P(w)$.
\end{Rem}
The purpose of this note is to prove Theorem~\ref{IrregularHLS}. We will provide a detailed proof for the operator~$\IA_\alpha$, the proof for~$\tI_\alpha$ follows the same steps with several shortcuts possible. The proofs in~\cite{NakaiSadasue2012} and~\cite{Watari1964} establish an \emph{a priori} stronger inequality
\eq{
\|\I_{\alpha}[F]^*\|_{L_q} \lesssim \|F\|_{L_p},
}
here the star sign means the martingale maximal function:~$G^*(w) = \sup_n |G_n(w)|$; the notation~$A \lesssim B$ is short for~$A\leq CB$ where the constant~$C$ is uniform with respect to the parameters that are clear from the context. We do not know whether a similar maximal inequality is true for the operator~$\IA_\alpha$ (it does not follow from Theorem~\ref{IrregularHLS} since~$\IA_\alpha$ is not a martingale transform). Our proof will go as follows. We will establish the endpoint inequalities
\alg{
\label{WeakOne}\|\IA_\alpha[F]\|_{L_{q,\infty}} \lesssim \|F\|_{L_1},\quad q = \frac{1}{1-\alpha};\\
\label{WeakTwo}\|\IA_\alpha [F]\|_{L_{\infty}} \lesssim \|F\|_{L_{p,1}},\quad p = \frac{1}{\alpha}
}
and then interpolate. This way of proving the Hardy--Littlewood--Sobolev inequality is related to O'Neil's inequality, see p. 38 in~\cite{Peetre1976}.

Before we pass to details, we briefly describe the motivation. The Hardy--Littlewood--Sobolev inequality was designed by Sobolev in~\cite{Sobolev1938} as an instrument to prove what is now called the Sobolev embedding theorem. In the last two decades, there was a strong interest in the so-called Bourgain--Brezis inequalities (originated in~\cite{BourgainBrezis2002}). These inequalities somehow extend the Sobolev embedding theorem in the limit case~$p=1$ to the setting of more complicated differential operators (we refer the reader to the survey~\cite{VanSchaftingen2014} for more information); there are still many open questions in this field. In~\cite{ASW2018}, a probabilistic model for Bourgain--Brezis inequalities was suggested; it appeared that, in a sense, these inequalities naturally extend Watari's theorem (for uniform not necessarily dyadic filtrations) to the limit case~$p=1$ by imposing linear constraints on each step of the martingale (see~\cite{Stolyarov2019} as well). It is now natural to try to find the way back from discrete uniform martingales to the classical Euclidean setting, and the extension to the class of arbitrary martingales seems desirable. We have not found such an extension for a simpler phenomenon (that is the Hardy--Littlewood--Sobolev inequality) in the literature, and hope that the present paper fills this gap; it might be thought of as the first step towards the theory of Bourgain--Brezis inequalities for irregular martingales.   

We are grateful to Adam Osekowski, Pavel Zatitskii, and Ilya Zlotnikov for attention to our work.   

\section{Functional analysis preparation}
Let us first comment on our use of the~$L_p$ norms. The scalars in our considerations are always real, however, the reasonings work for the complex case as well. We understand the operator~$\IA_\alpha$ in the following sense: one takes a summable random variable~$F=F_\infty$, constructs the martingale~$F$ (which we denote by the same letter) by the formula
\eq{
F_n  = \E(F_{\infty}\mid \F_n),
}
then computes the sum~\eqref{IA}, and obtains the random variable~$\IA_\alpha[F]$. So, we consider our operators as linear mappings between function spaces. The classical martingale mappings do not differ from these when~$p > 1$ (by Doob's convergence theorem). Our standpoint allows to verify the boundedness of linear operators on simple functions (by a simple function we mean a random variable that is~$\F_n$-measurable for some large~$n$) and martingales since any summable random variable allows approximation in~$L_1$ by simple random variables. One may derive the almost sure convergence of the series~\eqref{tI} and~\eqref{IA} and corresponding inequalities for arbitrary~$L_p$ martingales~$F$ from Theorem~\ref{IrregularHLS} by a routine limiting argument.

We will be using Lorentz spaces; we refer the reader to~\cite{Grafakos2008Classical} for a detailed exposition. A Lorentz space~$L_{p,q}, 1 \leq p < \infty$, is the space of measurable random variables~$f$ such that the quasi-norm
\eq{\label{Lorentz}
\|f\|_{L_{p,q}} = \Big\|t\big(P(|f| > t)\big)^\frac{1}{p}\Big\|_{L_q(\mathbb{R_+}, dt/t)}
}
is finite (here $dt$ is the Lebesgue measure on the line). The quasi-norm above may violate the triangle inequality if~$q\ne p$, however, the Lorentz space~$L_{p,q}$ may be equipped with an equivalent norm when~$p > 1$. In particular, it may be treated as a Banach space. We will use this fact several times referring to it as "Lorentz spaces are normable". 

\begin{Le}\label{DualityLorentz}
If~$q \in (1,\infty)$\textup, then~$(L_{q, \infty})^* \supset L_{q', 1}$\textup, where~$q' = \frac{q}{q-1}$ is the conjugate exponent to $q$. 
\end{Le}
The lemma states that each function~$f \in L_{q', 1}$ defines a continuous linear functional on $L_{q, \infty}$ according to the standard formula $g\mapsto \E fg$.  The lemma is a consequence of the formula
\eq{
L_{q', 1}^* = L_{q, \infty},\quad q > 1,
}
(see Theorem~$1.4.17$ in~\cite{Grafakos2008Classical}) and the fact that the second dual contains the Banach space itself (here we use that Lorentz spaces are normable). The complete description of the dual space~$(L_{q, \infty})^*$ is given in~\cite{Cwikel1975}.

\begin{Le}\label{FirstWeakType}
Inequality~\eqref{WeakOne} is true.
\end{Le}
\begin{Le}\label{SecondWeakType}
Inequality~\eqref{WeakTwo} is true.
\end{Le}
The proofs of the two lemmas above are presented in Section~\ref{S3} below. Theorem~\ref{IrregularHLS} follows from them by standard interpolation, e.g. by Theorem~$1.4.19$ in~\cite{Grafakos2008Classical}. We end this section with some "soft" functional analysis preparation to the proofs of the lemmas. First, we would like to dualize inequality~\eqref{WeakTwo}, and for that we need to compute the conjugate operator~$(\IA_\alpha)^*$. We do not care about domains of operators and convergence since we are allowed to work with simple functions and martingales. Let~$M_n$ be the operator of multiplication by~$b_n^\alpha$, let us also write~$E_n F = \mathbb{E}(F \mid \F_n)$ for brevity. The Riesz potential~$\IA_\alpha$ can be rewritten within new terms as
\eq{
\IA_\alpha[F] = \sum \limits_{n=1}^{\infty} M_n(E_n-E_{n-1}). 
}
Note that the operators~$M_n$ and~$E_n$,~$n=1,\ldots,\infty$, are self-adjoint. Therefore, the formal conjugate operator to $\IA_\alpha$ may be expressed as
\eq{\label{star} 
(\IA_\alpha)^* = \sum \limits_{n=1}^{\infty} (E_n-E_{n-1})M_n.
}
We understand this identity in the sense that the formula
\eq{\label{Conjugation}
\E \IA_\alpha[F] G = \E F(\IA_\alpha)^*[G]
}
is true provided both functions~$F$ and~$G$ are simple. There is a peculiarity here: the conjugate operator does not map simple functions to simple functions; however, as we will prove in the next section (see formulas~\eqref{J1},~\eqref{J2}, and~\eqref{Uniform}), it maps simple functions to bounded ones, which allows to work with the formula~\eqref{Conjugation}. Therefore, Lemma~\ref{SecondWeakType} is reduced to the inequality
\eq{\label{WeakTwoDual}
\|(\IA_{\alpha})^*[G]\|_{L_{q,\infty}} \lesssim \|G\|_{L_1},\quad q = \frac{1}{1-\alpha},
}
via Lemma~\ref{DualityLorentz}. 

Thus, we are left with proving the~$L_1\to L_{q,\infty}$ boundedness of two operators (namely,~$\IA_\alpha$ and~$(\IA_\alpha)^*$). We end our preparation with a simple lemma about operators from~$L_1$ to a Banach space, which is merely a manifestation of the principle that the extremal points in the unit ball of the space of measures are the delta measures.
\begin{Def} 
A function~$F$ is called atomic provided it is a scalar multiple of a characteristic function of an atom~$w\in\A(\F_n)$ for some~$n$.
\end{Def}
\begin{Le}\label{DeltaMeasure}
Let $T$ be a linear operator initially defined on the set of simple functions and mapping them to some Banach space $X$. Assume there exists~$c > 0$ such that for any atomic $F$ the inequality
\eq{
\|TF\|_X \leq c\|F\|_{L_1}
}
holds true. Then\textup, the same estimation  is valid for any simple function $F$.
\end{Le}
\begin{proof}
Let~$F$ be an arbitrary simple function. 
Without loss of generality, let~$F$ be~$\F_N$-measurable. Then, we may write
\eq{
F = \sum\limits_{k=1}^{\infty} a_k \chi_{w_k},
}
where~$\A(\F_N) = \{w_k\mid k \in \mathbb{N}\}$ (possibly, there is only a finite number of atoms in~$\F_N$). Let~$g_k := a_k \chi_{w_k}$, note that the~$g_k$ are atomic. It remains to use the triangle inequality
\eq{
\|TF\|_X = \Big\|\sum \limits_{k=1}^{\infty} Tg_k\Big\|_X  \leq \sum\limits_{k=1}^{\infty} \|Tg_k\|_X \leq c\sum \limits_{k=1}^{\infty} \|g_k\|_{L_1} = c\|F\|_{L_1}.
}
\end{proof}
\section{Two weak type inequalities}\label{S3}
\begin{proof}[Proof of Lemma~\ref{FirstWeakType}.]
According to Lemma~\ref{DeltaMeasure} and the fact that Lorentz spaces are normable, it suffices to verify~\eqref{WeakOne} for atomic functions~$F$ only. Consider a sequence of atoms
\eq{
\Omega = w_0 \supset w_1 \supset \ldots \supset w_N,\quad w_n \in \A(\F_n),
} 
and an atomic function $F = a \chi_{w_N}$. Let~$P(w_n) = r_n$, for~$n = 0, \ldots, N$. Without loss of generality, we may assume
\eq{
F_n = \frac{1}{r_n} \chi_{w_n},\quad n = 0,1,\ldots, N,
}
i.e. we set~$a = r_N^{-1}$. Now let us write the action of~$\IA_\alpha$ explicitly:
\eq{\label{action} 
\IA_\alpha[F](x) = \sum \limits_{k=1}^{n} r_k^\alpha \left( \frac{1}{r_k} - \frac{1}{r_{k-1}}\right) - b_{n+1}^\alpha(x) \frac{1}{r_n},\qquad x \in w_n \setminus w_{n+1},\quad n = 1,\ldots, N-1.
}
In the case~$x\in w_N$, the formula is slightly simpler:
\eq{
\IA_\alpha[F](x) = \sum \limits_{k=1}^{N} r_k^\alpha \left( \frac{1}{r_k} - \frac{1}{r_{k-1}}\right).
}
We start with the estimate
\eq{
r_k^\alpha \left( \frac{1}{r_k} - \frac{1}{r_{k-1}}\right) = r_k^\alpha \int \limits_{r_k}^{r_{k-1}} \frac{1}{x^2}dx \leq \int \limits_{r_k}^{r_{k-1}} x^{\alpha-2}dx.
}
Using formula \eqref{action} and the estimate~$b_{n+1}(x) \leq r_n$ for~$x\in w_{n}\setminus w_{n+1}$, we obtain 
\eq{\label{PointwiseEstimate}
\Big|\IA_\alpha[F](x)\Big| \leq  \int \limits_{r_{n}}^1 x^{\alpha - 2} dx + r_n^{\alpha-1}= \frac{1 - r_{n} ^{\alpha-1}}{\alpha - 1} +r_n^{\alpha-1} \lesssim r_{n}^{-\frac{1}{q}}
}
for~$x\in w_n\setminus w_{n+1}$. For~$x\in w_N$, we also have~$|\IA_\alpha[F]| \lesssim r_N^{-\frac{1}{q}}$.

Recall that~\eqref{WeakOne} means that for every~$\lambda > 0$
\eq{\label{WeakTypeExplained}
P\big(|\IA_\alpha [F]| > \lambda\big) \lesssim \lambda^{-q}
}
since we have~$\|F\|_{L_1} = 1$. It follows from~\eqref{PointwiseEstimate} that if~$|\IA_\alpha [F](x)| > \lambda$, then~$x\in w_n$ with~$\lambda \lesssim r_n^{-\frac{1}{q}}$, which immediately leads to~\eqref{WeakTypeExplained}.
\end{proof}

\begin{proof}[Proof of Lemma~\ref{SecondWeakType}.]
It suffices to verify~\eqref{WeakTwoDual} for an atomic function~$G$. Let~$G = r_N^{-1}\chi_{w_N}$ as in the proof of the previous lemma. We split the operator $(\IA_\alpha)^*$ given by formula~\eqref{star}, into two parts:
\alg{\label{J1}\J^1 = \sum_{n=1}^{N} (E_n-E_{n-1})M_n,\\
\label{J2}\J^2 = \sum_{n>N} (E_n-E_{n-1})M_n.
}
It suffices to prove the inequalities
\alg{\label{J1bound}\sup \limits_{\lambda > 0} \lambda \Big(P\big(|\J^1[G]| > \lambda\big)\Big)^\frac{1}{q} \lesssim 1;\\
\label{J2bound}\sup \limits_{\lambda > 0} \lambda \Big(P\big(|\J^2[G]| > \lambda\big)\Big)^\frac{1}{q} \lesssim 1.
}

We begin with the estimate for the operator $\J^1$. Note that for~$n \leq N$, the equality $M_nG = r_n^\alpha G$ holds true, which helps to rewrite~$\J^1[G]$ in the following way (we use the same notation as in the proof of the previous lemma:~$w_n$ is the atom of~$\F_n$ containing~$w_N$,~$n \leq N$, we also postulate~$r_{N+1} = 0$ and~$w_{N+1} = \varnothing$):
\eq{
\J^1[G]  = \sum \limits_{n=1}^N r_n^\alpha(G_n - G_{n-1}) =-r_1^\alpha G_0 +  \sum \limits_{n=1}^{N} G_n(r_n^\alpha - r_{n+1}^\alpha) = -r_1^\alpha G_0 + \sum \limits_{n=1}^N \frac{r_n^\alpha - r_{n+1}^\alpha}{r_n} \chi_{w_n}.
}
In particular,
\eq{
\J^1[G]|_{w_{n} \setminus w_{n+1}} =-r_1^\alpha  + \sum \limits_{k=1}^{n} \frac{r_k^\alpha - r_{k+1}^\alpha}{r_k}
}
for any~$n \leq N$. Similar to the proof of Lemma~\ref{FirstWeakType}, we use the inequality
\eq{
\frac{r_k^\alpha - r_{k+1}^\alpha}{r_k} = \frac{1}{r_k} \int \limits_{r_{k+1}}^{r_k} \alpha x^{\alpha-1} dx \leq \alpha \int \limits_{r_{k+1}}^{r_k} x^{\alpha-2} dx
}
to obtain the pointwise bound
\eq{
|\J^1[G](x)| \leq 1+ \alpha \sum\limits_{k=1}^{n-1} \int \limits_{r_{k+1}}^{r_k} x^{\alpha-2} dx + r_n^{\alpha-1}\lesssim r_{n}^{-\frac{1}{q}},\qquad x\in w_{n}\setminus w_{n+1}.
}
This means that the inequality~$|\J^1[G](x)| > \lambda$ does indeed hold only within a set of size $O(\lambda^{-q})$, and we have proved~\eqref{J1bound}.

Now consider the operator $\J^2$. In fact, we will show the inequality
\eq{\label{Uniform}
\|\J^2[G]\|_{L_\infty}\lesssim r_N^{\alpha-1}
}
in this case. Since the function $\J^2[G]$ vanishes outside~$w_N$, this implies~\eqref{J2bound}. 

Consider a point $x$ lying within the atoms~$w_N \supset w_{N+1} \supset \ldots$ whose probabilities are~$P_N = r_N, P_{N+1}, \ldots$ respectively; we assume~$w_n \in \A(\F_n)$ as usually. Then, for all $n > N$, and~$x\in w_n$,
\alg{E_nM_nG(x) = M_nG(x) = \frac{P_n^{\alpha}}{r_N},\\
\frac{P_n^{\alpha+1}}{r_NP_{n-1}} \leq E_{n-1}M_nG(x) \leq \frac{P_{n-1}^\alpha}{r_N}.
}
Therefore,
\eq{
r_N^{-1} \sum\limits_{n > N} (P_n^\alpha - P_{n-1}^\alpha)\leq  \J^2[G](x) \leq r_N^{-1}\sum \limits_{n > N} P_n^\alpha\big(1 - \frac{P_n}{P_{n-1}}\big).
}
Using the inequality
\eq{
P_n^\alpha\big(1 - \frac{P_n}{P_{n-1}}\big) = P_n^{\alpha+1} \int \limits_{P_n}^{P_{n-1}} \frac{1}{x^2}dx \leq \int \limits_{P_n}^{P_{n-1}} x^{\alpha - 1} dx,
}
we obtain the bound
\eq{\label{Above}
\J^2[G](x) \leq r_N^{-1}\int \limits_0^{P_N} x^{\alpha - 1} dx = \frac{P_N^\alpha}{r_N\alpha} = \frac{r_N^{\alpha-1}}{\alpha}.
}

The bound from below is even simpler:
\eq{\label{Below}
\J^2[G](x) \geq r_N^{-1}\sum \limits_{n > N} (P_n^\alpha - P_{n-1}^\alpha) = -r_N^{-1}P_N^\alpha = -r_N^{\alpha-1}.
}
Since both estimates~\eqref{Above} and~\eqref{Below} hold true for any~$x \in w_N$, we have verified~\eqref{Uniform}.
\end{proof}

\bibliography{mybib}{}
\bibliographystyle{amsplain}

%
%
%
%

St. Petersburg State University, Department of Mathematics and Computer Science;

d.m.stolyarov at spbu dot ru,

jarcev.v.2017 at list dot ru.

\end{document}